\documentclass [12pt,a4paper] {amsart}

\usepackage[utf8]{inputenc}
\usepackage[T1]{fontenc}
\usepackage{lmodern}
\usepackage{microtype}

\usepackage {amssymb}
\usepackage {amsmath}
\usepackage {amsthm}
\usepackage{mathtools}
\usepackage{mathdots}
\usepackage{xfrac}
\usepackage{bm}

\usepackage{url}
 
\usepackage[titletoc,toc,title]{appendix}

\usepackage{enumitem}
\usepackage{hyperref}
\hypersetup{
    colorlinks=true,
    linkcolor=blue,
    filecolor=magenta,      
    urlcolor=cyan,
    citecolor=magenta
}
\usepackage{cleveref}

\usepackage{tikz}
\usetikzlibrary{arrows}

\DeclareMathOperator {\td} {td}

\DeclareMathOperator {\im} {Im}

\DeclareMathOperator {\alg} {alg}

\DeclareMathOperator {\seq} {\subseteq}
\DeclareMathOperator {\C} {\mathbb{C}}
\DeclareMathOperator {\R} {\mathbb{R}}

\DeclareMathOperator {\Z} {\mathbb{Z}}
\DeclareMathOperator {\Q} {\mathbb{Q}}

\newcommand{\s}{\mathbb{S}}
\newcommand{\N}{\mathbb{N}}

\newcommand{\vbar}{\ensuremath{\mathbf{v}}}
\newcommand{\sbar}{\ensuremath{\mathbf{s}}}

\newcommand{\zbar}{\ensuremath{\mathbf{z}}}
\newcommand{\wbar}{\ensuremath{\mathbf{w}}}

\newcommand{\re}{\mathrm{Re}}
\newcommand{\Cx}{\mathbb{C}^{\times}}

\newcommand{\albar}{{\ensuremath{\boldsymbol{\alpha}}}}
\newcommand{\bebar}{{\ensuremath{\boldsymbol{\beta}}}}

\theoremstyle {plain}
\newtheorem {theorem}{Theorem}[section]

\newtheorem {proposition} [theorem] {Proposition}

\newtheorem {conjecture} [theorem] {Conjecture}
\newtheorem {claim} {Claim}[theorem]

\theoremstyle {definition}

\newtheorem {definition}[theorem]{Definition}

\newtheorem{exercise} [theorem] {Exercise}

\theoremstyle {remark}
\newtheorem {remark} [theorem] {Remark}

\newenvironment{claimproof}{\bgroup\begin{proof}}{\end{proof}\egroup}

\calclayout

\title[Exponential sums and Exponential Closedness]{Exponential sums equations and the Exponential Closedness conjecture}
\author{Vahagn Aslanyan}

\email{Vahagn.Aslanyan@manchester.ac.uk}
\address{Department of Mathematics, University of Manchester, Oxford road, Manchester M13 9PL, UK}

\author{Francesco Gallinaro}
\email{francesco.gallinaro@dm.unipi.it}
\address{Dipartimento di Matematica, Universit\`a di Pisa, Largo Bruno Pontecorvo 5, 56127 Pisa, Italy.}

\begin{document}

\vspace*{-3cm}



\subjclass[2020]{}

\thanks{VA was supported by EPSRC Open Fellowship EP/X009823/1. FG was supported by project PRIN 2022 ``Models, sets and classifications'', prot. 2022TECZJA. For the purpose of open access, the authors have applied a Creative Commons Attribution (CC BY) licence to any Author Accepted Manuscript version arising from this submission.}

\maketitle
\begin{abstract}
This is an expository paper aiming to introduce Zilber's Exponential Closedness conjecture to a general audience. Exponential Closedness predicts when (systems of) equations involving addition, multiplication, and exponentiation have solutions in the complex numbers. It is a natural statement at the boundary between complex geometry and algebraic geometry. While it is open in full generality, many special cases and variants have been proven in the last two decades. 

In the first part of the paper we give a proof of a special case of the conjecture, namely, we show that exponential sums in a single complex variable, such as $e^{\sqrt{2}z}+3i e^{5z}+\pi$, always have infinitely many zeroes. This follows from well-known results in the literature and can actually be proven by standard complex analytic tools. We present a proof, motivated by the ideas of Zilber and Gallinaro, which uses simplified versions of some ingredients of more advanced techniques in the area, thus allowing the reader to explore these advanced proofs on a simple example. Our approach uses only elementary techniques (from standard undergraduate algebra and complex analysis), and the proofs of some of the main ingredients appear to be new. Moreover, these ingredients come from various areas of mathematics, such as functional transcendence and complex geometry, and are a useful way of introducing some classical concepts in these areas to the reader. 

In the second part we explain the Exponential Closedness conjecture and discuss some known special cases, not least the so-called ``raising to powers''  case (due to Gallinaro) which generalises the above-mentioned result and whose proof uses advanced versions of the tools presented in the first part of the paper.
\end{abstract}

\section{Introduction}

One of the most fundamental problems in algebra, and indeed in mathematics, is to understand whether a given equation has a solution. First, we remark that it does not make sense to simply ask if an equation has a solution. For instance, if we ask ``Does the equation $2x=1$ have a solution?'', then those who are familiar rational numbers will answer affirmatively, while those who know what integers are but not rational numbers, will give a negative answer. Therefore, it is important to to specify \textit{where} we are looking for solutions. Thus, the equation $2x=1$ has a solution in $\Q$ but not in $\Z$. 

The simplest equations usually considered in undergraduate algebra are those made up from the two basic arithmetic operations, addition and multiplication. These are known as polynomial equations, for they are of the form $p(x)=0$ where $p(X) = a_nX^n+a_{n-1}X^{n-1}+\ldots+a_1X+a_0$ is a polynomial. Polynomial equation may or may not have solutions in a given set of numbers. For instance, the equation $x^2+1=0$ has no real solution but has two complex solutions. A space where all non-trivial polynomial equations have solutions is the field of complex numbers. This is known as the \textit{Fundamental Theorem of Algebra}. From now on we will be concerned with solvability of equations in the complex numbers and will not state this explicitly every time. 

In this paper we look at more advanced equations involving not only addition and multiplication, but also complex exponentiation, and we want to understand when such equations have solutions in the complex numbers. First, recall that for a complex number $c = a+bi$, where $a$ and $b$ are real numbers, the exponential of $c$ is defined as $e^c = e^a (\cos b + i \sin b)$ where $e^a$ is the real exponential. We will use $z$ to denote a complex variable. Thus, we want to know whether equations such as $2e^{e^z}+e^{z^2}+3z^5+i=0$ have solutions in $\C$. More generally, we want to understand when systems of such equations in several variables have solutions. This question is open in full generality, however, there is a conjecture predicting when  systems of exponential equations should have solutions. We will discuss this conjecture shortly, but first we describe a certain class of exponential equations that we study and show their solvability in this paper. These are the so-called \textit{exponential sums equations}. An exponential sum is an expression of the form $\sum_{k=0}^n c_k e^{r_kz}$ where $c_k \in \C$ and $r_k \in \R$ for all $k$. Note that when each $r_k$ is a positive integer, denoting $e^z$ by $w$, we can write the above exponential sum as a polynomial in $w$, namely $\sum_{k=0}^n c_k w^{r_k}$. In this particular case the existence of a zero of the exponential sum is equivalent to the existence of a non-zero root of the corresponding polynomial, which is an easy problem to solve (using the fundamental theorem of algebra). In general, when the exponents are not necessarily integers, we can still informally think of an exponential sum as a sum of real powers of $w$. This is not precise since, given a real number $r$, the function $w^r$ is defined as $e^{r\log w}$ which is a multi-valued function due to the infinitely many branches of complex logarithm. Of course, writing $e^{rz}$ instead of $w^r$ takes care of this ambiguity, but thinking of raising to real powers rather than exponentials allows us to interpret our question as a generalisation of the fundamental theorem of algebra where the exponents of the variable are allowed to be real numbers. Note also that in number theory exponential sums usually refer to somewhat different expressions which will not be considered here.

In the first part of the paper we prove the following result.

\begin{theorem}\label{thm: exp sums one var}
    Let $n$ be a positive integer, $r_n > \ldots > r_0$ be arbitrary real numbers and $c_0, c_1,\ldots,c_n$ be non-zero complex numbers. Then the equation 
    \begin{equation}\label{eqn: sum e^{r_kz}=0}
        \sum_{k=0}^n c_k e^{r_kz} = 0 
    \end{equation}
     has infinitely many solutions.
\end{theorem}

This theorem is rather well known: it can be proven quite easily using standard tools from complex analysis (see Exercise~\ref{exercise: Hadamard})\footnote{Thanks to Aleksei Kulikov and Alex Wilkie for pointing this out.}, and it is a special case of the much more general results of Henson--Rubel \cite{henson-rubel}, Gallinaro \cite{Gallinaro-exp-sums-trop}, and Mantova--Masser \cite{mantova-masser}. Our proof is based on ideas of Zilber \cite{Zilb-exp-sum-published,ZilbExpSum} and Gallinaro \cite{Gallinaro-exp-sums-trop} and thus allows the reader to explore these ideas on a simple example. Nevertheless, there are some novelties in the way we present the main ingredients of the proof. Thus, the focus here is on these ingredients in the proof of Theorem~\ref{thm: exp sums one var} rather than the theorem itself.

The key idea (originally due to Khovanskii) of the proof is that if we approximate the real exponents $r_k$ in \eqref{eqn: sum e^{r_kz}=0} by rational numbers, then the resulting equation will be an algebraic equation and thus will have solutions. So it is natural to expect that a limit of these solutions as we let the exponents tend to the $r_k$'s would be a solution to \eqref{eqn: sum e^{r_kz}=0}. The problem is that these approximate solutions may escape off to infinity at the limit, breaking the whole argument. Thus, the main difficulty is to bound these solutions. More precisely, to generate a new sequence from this sequence of solutions, the limit of which gives us the desired zero of the exponential sum. 

The main ingredients of our proof include: estimates on the roots of a polynomial in terms of its coefficients, density of integral multiples of certain vectors in $\R^n$ modulo $\Z^n$, algebraic independence of certain exponential functions which is a special case of Ax's famous theorem from functional transcendence, and a particular case of the open mapping theorem from complex analytic geometry. We provide a full proof in each case; even though all of these statements are well-known results (mostly as a special case of a more complicated theorem), only one of the proofs is directly borrowed from the literature the rest containing some potentially new ideas or techniques. We stress that all of these proofs require only basic knowledge of undergraduate mathematics and should be accessible to a wide audience. Moreover, these ingredients are simple versions of classical statements in various areas of mathematics (functional transcendence, complex geometry, equidistributed sequences) and, along with their elementary proofs, can be studied on their own right and independently from each other. This will familiarise the reader with some ideas and tools in the aforementioned areas of mathematics. We also have some exercises that will help the reader further explore some of the presented topics.

The second part of the paper is about solvability of systems of exponential equations in several variables. Here the \textit{Exponential Closedness} conjecture, due to Zilber \cite{Zilb-pseudoexp,Zilb-exp-sum-published}, predicts that \textit{rotund} and \textit{free} systems must have complex solutions. Rotundity and freeness are geometric properties of the system under consideration which ensure that the system, and every subsystem, is not \textit{overdetermined}. Intuitively, a system is overdetermined if it contains more equations than variables. We do not expect such systems to have any solutions, e.g. $z^2=1, e^z=2$ has no complex solutions. In the second part of the paper we define freeness and rotundity and give a precise statement of the Exponential Closedness conjecture. Although it is open, many special cases have been established in recent years, a survey of which is also given in this paper. A special emphasis is put on the so-called \textit{raising to powers} case of the conjecture, which is a much more general version of Theorem~\ref{thm: exp sums one var} and was proved by Gallinaro in \cite{Gallinaro-exp-sums-trop} (generalising earlier results of Zilber \cite{ZilbExpSum,Zilb-exp-sum-published}). 

The second part of the paper is somewhat more advanced than the first part (requiring familiarity with basic algebraic geometry) but should still be accessible to graduate and advanced undergraduate students in mathematics.

\section{Exponential sums in one variable}

In this section we prove Theorem~\ref{thm: exp sums one var}. Before presenting our proof, we point out how it can be obtained by standard complex analytic tools.

\begin{exercise}\label{exercise: Hadamard}
\begin{itemize}
\item[]
    \item Adapt the proof of \cite[Corollary 2.4]{marker-remark-pseudoexp} and deduce Theorem~\ref{thm: exp sums one var} from Hadamard's factorisation theorem (a statement of which can be found in the same paper).
    \item Give an alternative proof of Theorem~\ref{thm: exp sums one var} based on the following ideas. If an exponential sum has no zeroes, then it can be written as $e^{f(z)}$ for some entire function $f$. Show that the real part of $f$ is then bounded by a linear function of $|z|$. Use the Borel-Caratheodory theorem to get a similar bound for $|f(z)|$. Use Liouville's theorem to conclude that $f$ must be linear, which then should imply that $f$ has a single exponential term. Finally, adapt the argument to show that an exponential sum with at least two terms has inifnitely many zeroes.

    Can these ideas be used to prove a more general result, e.g. when the $r_k$'s in \eqref{eqn: sum e^{r_kz}=0} are complex numbers?
\end{itemize}
\end{exercise}

The first approach was suggested by Aleksei Kulikov and the second by Alex Wilkie. While Hamadamard's factorisation theorem is often omitted from basic complex analysis courses, the second approach uses more basic tools usually covered in such courses.

Now we turn to our proof which follows the ideas of Zilber and Gallinaro. We establish some auxiliary results first. These come from various areas of mathematics and are interesting in their own right.

\subsection{Bounds on roots of a polynomial}

Here we show how the roots of a polynomial can be bounded in terms of its degree and coefficients.

\begin{proposition}\label{prop: bound on roots}
    Let $N>n>m$ be positive integers and let $p(w) = c_N w^N + \sum_{k=m}^n c_kw^k + c_0$ be a polynomial with $c_0 \neq 0,~ c_N \neq 0$. Define \[ R_1:=\max\left\{ 1, \frac{1}{|c_N|}\left(|c_0|+\sum_{k=m}^{n}|c_k|\right) \right\} \text{ and } R_0 :=\min \left\{ 1, \frac{|c_0|}{|c_N|+\sum_{k=m}^n|c_k|} \right\}.\] Then all roots of $p(w)$ satisfy $R_0^{\frac{1}{m}} \leq |w| \leq R_1^{\frac{1}{N-n}}$.
\end{proposition}
\begin{proof}
     Observe that when $k\leq n$ and $|w|>R_1^{\frac{1}{N-n}}$ we have $|w|^{n} \geq |w|^{k}$ and \[ |c_N| |w|^{N} = |c_N| |w|^{N-n} |w|^n > R_1 |c_N| |w|^{n} \geq |c_0|+ \sum_{k=m}^{n} |c_k| |w|^{k} \geq \left| c_0+\sum_{k=m}^{n} c_k w^{k} \right| \] where the last inequality follows from the triangle inequality. Thus, if $|w|>R_1^{\frac{1}{N-n}}$ then $p(w)\neq 0$.

     Similarly, if $|w|<R_0^{\frac{1}{m}}$ then for all $k\geq m$ we have $|w|^{k} \leq |w|^m<R_0$ and 
     \[  \left| \sum_{k=m}^n c_k w^{k} + c_Nw^N \right| \leq  \sum_{k=m}^n |c_k| |w|^{k} + |c_N| |w|^N  < R_0 \left( \sum_{k=m}^n |c_k| +|c_N| \right) \leq |c_0|. \]
     So for such $w$ we have $p(w)\neq 0$.
\end{proof}

\subsection{Density of certain sequences in tori}

It is well known and quite easy to show that if $r\in \R$ is a rational number then the integral multiples of $r$ reduced modulo $1$ form a finite set. On the other hand, if $r$ is irrational, then its integral multiples modulo $1$ form a dense subset of $[0,1)$. Here reducing modulo $1$ just means that we replace a real number by its fractional part: for a real number $x$ we define the \textit{integral part} of $x$, denoted $[x]$, to be the largest integer not exceeding $x$ and we let the \textit{fractional part} of $x$, written $\{ x \}$, be $x-[x]$. This notation may be confusing as braces are normally used to denote sets. To avoid confusion, we do not refer to fractional part outside this subsection and here $\{ x \}$ always denotes the fractional part.

\begin{exercise}
    Let $r\in \R$. Prove that the set of all fractional parts $\{ nr \}$ as $n$ ranges over the integers is dense in $[0,1)$ if and only if $r\notin \Q$.
\end{exercise}

The following is a generalisation of this result to higher dimensions, written in terms of exponentials. 

\begin{proposition}\label{prop: density}
    Let $r_1,\ldots,r_n \in \R$ be real numbers such that the numbers $1, r_1, \ldots, r_n$ are linearly independent over $\Q$. Then the set $\{ ( e^{2\pi i r_1 m}, \ldots, e^{2 \pi i r_n m}): m \in \Z \}$ is dense in $\s_1^n$.
\end{proposition}
\begin{proof}
We present a proof for $n=2$, which is due to Lettenmeyer and is taken from \cite{hardy-wright}. First observe that it suffices to show that the set of points $ P_m:=  (\{r_1m\}, \{r_2m\})$ with $ m \in \Z$, where $\{ x \} = x-[x]$ is the fractional part of $x$, is dense in $[0,1)^2$.

Let $\varepsilon>0$ be sufficiently small. We need to show that for every point of the unit square there is a $P_m$ within an $\varepsilon$ distance from that point.  Since the set of points $P_m$ is bounded, it must have a limit point, i.e. there is a point $P_{m_0}$ such that for infinitely many $m$ the distance between $P_m$ and $P_{m_0}$ is less than $\varepsilon$. We may assume without loss of generality that $m_0=1$.\footnote{Justifying this is left to the reader as an exercise.} 

Now, if there are $k, l$ such that $|P_1P_k|<\varepsilon$ and $|P_1P_l|< \varepsilon$ and $P_1,P_k,P_l$ are not collinear then the points in the lattice generated by these three points are among the $P_m$'s and every point in the unit square is within an $\varepsilon$ distance of one of these points. So we are done in this case.

Now suppose that for all $k, l$ with $|P_1P_k|<\varepsilon$ and $|P_1P_l|< \varepsilon$, the points $P_1,P_k,P_l$ are collinear. Recall that $k$ and $l$ can be taken arbitrarily large. Collinearity means that

\[ 0 = \begin{vmatrix}
r_1 & r_2 & 1\\
\{ kr_1 \} & \{ kr_2 \} & 1\\
\{ lr_1 \} & \{ lr_2 \} & 1
\end{vmatrix} =
\begin{vmatrix}
r_1 & r_2 & 1\\
[kr_1] & [kr_2] & k-1\\
[lr_1] & [lr_2] & l-1
\end{vmatrix} .
\]
However, $r_1, r_2, 1$ are $\Q$-linearly independent, hence 
\[  \begin{vmatrix}
[kr_1]  & k-1\\
[lr_1]  & l-1
\end{vmatrix} = 0. \]
Thus, \[ \frac{[kr_1]}{k-1} = \frac{[lr_1]}{l-1}. \]

Now we can fix $k$ and let $l$ tend to infinity. Then the left hand side is a constant rational number while the right hand side tends to $r_1$ which is irrational. This contradiction finishes the proof.    
\end{proof}

\begin{exercise}
    Generalise the above proof to arbitrary $n$.
\end{exercise}

\subsection{Algebraic independence of certain exponential functions}

\begin{definition}
    Let $f_1,\ldots,f_n$ be complex functions defined on some complex domain $U \seq \C$. We say that they are \textit{algebraically independent} over $\C$ if there is no polynomial $p(Y_1,\ldots,Y_n)\in \C[Y_1,\ldots,Y_n]$ for which $p(f_1(z),\ldots,f_n(z))$ is identically $0$ on $U$.
\end{definition}

For example, the functions $z$ and $z^2$ are algebraically dependent while $z$ and $e^z$ are algebraically independent.

\begin{exercise}
    Prove that $z$ and $e^z$ are algebraically independent.
\end{exercise}

\begin{proposition}\label{proposition: ax-schanuel type statement}
    Let $r_1,\ldots,r_n$ be real numbers such that $r_1,\ldots,r_n$ are $\Q$-linearly independent. Then the functions $e^{r_1z},\ldots,e^{r_nz}$ are algebraically independent over $\C$.
\end{proposition}
\begin{proof}
    Suppose there is a polynomial $p$ with $p(e^{r_1z},\ldots,e^{r_nz})=0$. Let \[ p(\wbar) = \sum_{(k_1,\ldots,k_n)} c_{(k_1,\ldots,k_n)} w_1^{k_1}\cdots w_n^{k_n}.\]
    Then \[ p(e^{r_1z},\ldots,e^{r_nz}) = \sum_{(k_1,\ldots,k_n)} c_{(k_1,\ldots,k_n)} e^{(k_1r_1+\ldots +k_nr_n)z} .\]
    Since $r_1,\ldots,r_n$ are linearly independent over $\Q$, for any two tuples $(k_1,\ldots,k_n)\neq (k'_1,\ldots,k'_n)$ we have $k_1r_1+\ldots +k_nr_n \neq k'_1r_1+\ldots +k'_nr_n$. We then see that for a list of pairwise distinct real numbers $t_1,\ldots,t_m$ and non-zero complex numbers $b_1,\ldots,b_m$ we have
    \[ b_1 e^{t_1z}+ \ldots + b_n e^{t_nz} = 0. \]
    Differentiating this equality $m$ times we get 
    \[  \begin{pmatrix}
1 & 1 &  \dots & 1 \\ 
t_1 & t_2  & \dots & t_m \\
\vdots & \vdots & \ddots & \vdots \\
t_1^m & t_2^m  & \dots & t_m^m
\end{pmatrix} \cdot \begin{pmatrix}
b_1e^{t_1z}\\ 
b_2e^{t_2z} \\
\vdots \\
b_me^{t_mz}
\end{pmatrix} = 0.
\] 
The matrix on the left is a Vandermonde matrix whose determinant does not vanish, for the $t_k$'s are pairwise distinct. This yields a contradiction which finishes the proof.
\end{proof}

This theorem is a special case of the famous Ax-Schanuel theorem which plays an important role in many problems in differential algebra, number theory, and complex geometry. We do not need the full theorem in this paper but we state it below for the convenience of the reader.

\begin{theorem}[Ax-Schanuel {\cite{Ax}}]\label{Axsch}
    Let $f_1,\ldots,f_n$ be complex holomorphic functions defined on some domain $U\seq \C$. Assume that no $\Q$-linear combination of these functions is constant. Then $\td_{\C}\C(f_1,\ldots,f_n, e^{f_1},\ldots,e^{f_n})\geq n$.
\end{theorem}

Here $\td$ stands for \textit{transcendence degree} which is the maximum number of algebraically independent elements from the given set.

\begin{exercise}
    Deduce Proposition~\ref{proposition: ax-schanuel type statement} from the Ax-Schanuel theorem.
\end{exercise}

\subsection{An open mapping}

A classical result in complex analysis states that every non-constant holomorphic function of a single complex variable is an open map, i.e. it maps open sets to open sets. This follows from Rouch\'e's theorem. There is an extension of this theorem to functions of several complex variables, known as the \textit{open mapping theorem}. A special case is formulated below. 

\begin{theorem}[Special case of the open mapping theorem]\label{thm: special open mapping}
    Let $U$ be a domain in $\C^n$ and let $F:U\to \C^n$ be an analytic map (i.e. all partial derivatives of $F$ exist and are finite at every point of $U$). Assume that for every $\wbar \in \C^n$ the fibre $F^{-1}(\wbar):= \{ \zbar \in U: F(\zbar)=\wbar \}$ is finite (including when it is empty). Then $F$ is an open map, i.e. it maps open subset of $U$ to open subsets of $\C^n$.
\end{theorem}

This can be proven using Rouch\'e's theorem for several variables, but we do not give a proof since we do not need the full result in this paper. What we need is the following statement, which can be proven easily using the above theorem. We give an alternative proof based on tools from single-variable complex analysis.

\begin{proposition}\label{prop: open map}
    Let $p(\wbar)$ be a polynomial and let $r_1,\ldots,r_n$ be $\Q$-linearly independent real numbers. Then the set \[\{ (w_1e^{-r_1z}, \ldots, w_ne^{-r_nz}): z\in \C,~ \wbar \in \C^n \text{with } p(\wbar) = 0 \}\] is open in $\C^n$.
\end{proposition}
\begin{proof}
    Let $z_0 \in \C$ and let $\wbar_0=(w_{01},\dots,w_{0n})\in (\Cx)^n$ with $p(\wbar_0) = 0$. Let $\albar=(\alpha_1,\dots,\alpha_n)$ be the point $\left( w_{01}e^{-r_1z_0},\dots, w_{0n}e^{-r_nz_0}\right).$ Then we have $$\wbar_0=\left( \alpha_1e^{r_1z_0},\dots, \alpha_ne^{r_nz_0}\right).$$
    Consider the one-variable holomorphic function $F:\mathbb{C} \rightarrow \mathbb{C}$ defined by $F(z)=p\left( \alpha_1e^{r_1z},\dots, \alpha_ne^{r_nz}\right)$. By Proposition \ref{proposition: ax-schanuel type statement}, the functions $e^{r_1z},\dots,e^{r_nz}$ are algebraically independent over $\mathbb{C}$, and therefore $F$ is not identically 0. Therefore, $z_0$ is an isolated zero of $F$. Then there is $\epsilon >0$ such that $z_0$ is the only zero of $F$ in the closed ball $B:=\overline{B(z_0,\epsilon)}$. In particular, $F$ is never 0 on the boundary $\partial B$; let $m_1:=\min \{ |F(z)| : z \in \partial B\}>0$. 

    We now write the polynomial $p(\wbar)$ as \[ p(\wbar) = \sum_{(k_1,\ldots,k_n)} c_{(k_1,\ldots,k_n)} w_1^{k_1}\cdots w_n^{k_n}.\] Let $S:=\left\{(k_1,\dots,k_n) \in \mathbb{N}^n : c_{(k_1,\dots,k_n)} \neq 0 \right\}$, and let $$m_2:=\max \left\{|c_{(k_1,\dots,k_n)}e^{r_1k_1z}\cdots e^{r_nk_nz}| : (k_1,\dots,k_n) \in S, z \in \partial B \right\}.$$ 

    Finally, consider the set $$\Omega:=\left\{\bebar \in \C^n : \left| \prod_{i=1}^n \beta_i^{k_i} - \prod_{i=1}^n \alpha_i^{k_i} \right| < \frac{m_1}{2|S|m_2} \,\, \text{ for all } (k_1,\dots,k_n) \in S \right\}.$$ As $S$ is finite, this is a finite intersection of open sets; it is therefore open, and it obviously contains $\albar$. Now, it suffices to prove the following claim.

    \textbf{Claim}: Every point in $\Omega$ has the form $ (w_1e^{-r_1z}, \ldots, w_ne^{-r_nz})$ for some $z \in \C$ and some root $\wbar$ with $p(\wbar) = 0$.

    \begin{claimproof}
        Let $\bebar \in \Omega$, and let $G_{\bebar}$ denote the one-variable function defined on $B$ by $z \mapsto p(\beta_1e^{r_1z}, \ldots, \beta_ne^{r_nz})$.

        Let $z \in \partial B$. Then we have:
        \begin{align*}
            |G_{\bebar}(z)-F(z)| &= \left| \sum_{(k_1,\dots,k_n) \in S} c_{(k_1,\dots,k_n)} \prod_{i=1}^n \beta_i^{k_i}e^{r_ik_iz} - \sum_{(k_1,\dots,k_n) \in S} c_{(k_1,\dots,k_n)} \prod_{i=1}^n \alpha_i^{k_i}e^{r_ik_iz}  \right| \\
                              &= \left| \sum_{(k_1,\dots,k_n) \in S} c_{(k_1,\dots,k_n)} \left( \prod_{i=1}^n \beta_i^{k_i}e^{r_ik_iz} -  \prod_{i=1}^n \alpha_i^{k_i}e^{r_ik_iz}\right) \right| \\
                              &= \left| \sum_{(k_1,\dots,k_n) \in S} c_{(k_1,\dots,k_n)}\prod_{i=1}^n e^{r_ik_iz} \left( \prod_{i=1}^n \beta_i^{k_i} - \prod_{i=1}^n \alpha_i^{k_i}\right) \right| \\
                              & \leq \sum_{(k_1,\dots,k_n) \in S} \left| c_{(k_1,\dots,k_n)} \prod_{i=1}^n e^{r_ik_iz} \right|\left| \prod_{i=1}^n \beta_i^{k_i} - \prod_{i=1}^n \alpha_i^{k_i} \right| \\
                              & \leq \sum_{(k_1,\dots,k_n) \in S} m_2 \cdot \frac{m_1}{2|S|m_2} \\
                              &= \frac{m_1}{2}\\
                              &<|F(z)|.
        \end{align*}        
        Therefore by Rouch\'e's theorem $G_{\bebar}$ has a zero in $B$. This means that there is $z_1 \in B$ such that $p(\beta_1e^{r_1z_1},\dots,\beta_ne^{r_nz_1})=0.$ Obviously $$\bebar=\left( \beta_1e^{r_1z_1}e^{-r_1z_1},\dots,\beta_ne^{r_nz_1}e^{-r_nz_1}\right)$$ so we are done.
    \end{claimproof}
    This finishes the proof of Proposition~\ref{prop: open map}.
\end{proof}
    
\begin{exercise}
    Give an alternative (much simpler) proof of Proposition~\ref{prop: open map} using Theorem~\ref{thm: special open mapping}.
\end{exercise}

\subsection{Proof of Theorem \ref{thm: exp sums one var}} We now prove the main theorem.

\textbf{Step 1.} First observe that by multiplying the equation by $e^{-r_0z}$ (which is never zero) we may assume $r_k>0$ for all $k>0$ and $r_0=0$. Let $0<t_1< \ldots< t_m$ be a maximal $\Q$-linearly independent subset of $\{ r_1,\ldots, r_n \}$. Then each $r_k$ can be written as a $\Q$-linear combination of the $t_1, \ldots, t_m$, i.e. $r_k = \sum_{j=1}^m s_{k,j}t_j$ with $s_{k,j}\in \Q$. We may substitute these expressions in \eqref{eqn: sum e^{r_kz}=0} and get an equation in terms of $e^{s_{k,j}t_jz}$. Let $d$ be a common denominator of all $s_{k,j}$'s. Replacing $z$ by $dz$ we see that \eqref{eqn: sum e^{r_kz}=0} can now be written as 
\begin{equation}\label{eqn: p(e^t_1z,...) = 0}
    p(e^{t_1z}, \ldots, e^{t_mz}) = 0
\end{equation}
where $p(w_1,\ldots,w_m)$ is a polynomial.

\textbf{Step 2.}  Let $q_{k,l}, 1\leq k \leq m, l\geq 0$ be rational numbers such that for each $k$ the sequence $q_{k,l}$ tends to $t_k$ as $l\to \infty$ and for each $l$ the inequalities $q_{m,l}>q_{m-1,l}>\ldots>q_{1,l}$ hold. 

For each $l$ consider the equation \[ p(e^{q_{1,l}z},\ldots,e^{q_{m,l}z}) = 0. \] Write $q_{k,l} = \frac{a_{k,l}}{b_l}$ with $a_{k,l}, b_l \in \N$. Here $b_l$ is a common denominator for $q_{1,l},\ldots,q_{m,l}$. Now we introduce a new variable $w:=e^{\frac{z}{b_l}}$ and get the equation
\begin{equation}\label{eqn: w_l poly}
   p( w^{a_{1,l}},\ldots, w^{a_{m,l}} )= 0.
\end{equation}

This is a non-trivial polynomial equation with respect to $w$, so it has complex solutions by the fundamental theorem of algebra. By Proposition \ref{prop: bound on roots}, there are two real numbers $0<R_0 \leq 1 \leq R_1$, depending only on the coefficients of $p$ (and not on $l$), such that all solutions of \eqref{eqn: w_l poly} satisfy $R_0^{\frac{1}{M_{0,l}}} \leq |w| \leq R_1^{\frac{1}{M_{1,l}}}$ where $M_{0,l}=\frac{b_lr_0}{2}$ and $M_{1,l}=\frac{b_l(r_n-r_{n-1})}{2}$.

For each $l$ we let $w_l$ be a solution of \eqref{eqn: w_l poly} and let $z_l \in \C$ be such that $e^{\frac{z_l}{b_l}} = w_l$. Since  $R_0^{\frac{1}{M_{0,l}}} \leq |w_l| \leq R_1^{\frac{1}{M_{1,l}}}$, we can conclude that \[ \frac{2}{r_0}\log R_0 = \frac{b_l}{M_{0,l}}\log R_0\leq \re(z_l) \leq \frac{b_l}{M_{1,l}}\log R_1 = \frac{2}{r_n - r_{n-1}}\log R_1.\] 
Thus, $\re(z_l)$ is bounded. If we also knew that $\im(z_l)$ is bounded, then $z_l$ would have a limit point $z^*\in \C$ which would then be a solution to \eqref{eqn: sum e^{r_kz}=0}. Unfortunately, the $z_l$'s may actually wander off to infinity (vertically), so we need a more sophisticated argument here.

\textbf{Step 3.} Let $\hat{z}_l := z_l \mod 2\pi i$, that is, $\hat{z}_l$ is a complex number such that $\hat{z}_l-z_l$ is an integral multiple of $2\pi i$ and $\im (\hat{z}_l) \in [0, 2\pi) $. Let $\hat{z}\in \C$ be a partial limit of the sequence $\hat{z}_l$.

\begin{claim}
    There is $\wbar\in \C^m$ such that $p(\wbar)=0$ and $|w_ke^{-t_k\hat{z}}|=1$ for all $k$.
\end{claim}
\begin{claimproof}
    We know that $ p(e^{q_{1,l}z_l},\ldots,e^{q_{m,l}z_l}) = 0$ and $\re(q_{k,l}z_l-t_k\hat{z}) \to 0$, hence $|e^{q_{k,l}z_l}e^{-t_k\hat{z}}|\to 1$ as $l\to \infty$. Thus, there is a sequence $\wbar_l$ such that $p(\wbar_l)=0$ and $|w_{k,l}e^{-t_k\hat{z}}|\to 1$ as $l\to \infty$. In terms of the above notation, we have $w_{k,l} = w_l^{a_{k,l}}$ and so $R_0^{\frac{2a_{k,l}}{r_0b_l}} \leq |w_{k,l}| \leq R_1^{\frac{2a_{k,l}}{(r_n-r_{n-1})b_l}}$. On the other hand $\frac{a_{k,l}}{b_l}\to t_k$, hence $|w_{k,l}|$ is bounded from above and from below by positive constants which do not depend on $l$. Thus, the sequence $\wbar_l$ has a limit point $\wbar^*\in \Cx$. Since $p(\wbar^l)=0$ and $p$ is continuous, $p(\wbar^*)=0$. Furthermore, $|w^*_ke^{-t_k\hat{z}}| =\lim_{l \to \infty} |w_{k,l}e^{-t_k\hat{z}}|=1$.
\end{claimproof}

\textbf{Step 4.} Proposition \ref{prop: density} tells us that the set  $\{ ( e^{-t_1(\hat{z}+2\pi i l)}, \ldots, e^{-t_m(\hat{z}+2\pi i l)}): l \in \Z \}$ is dense in $\s_1^m$. By Proposition \ref{prop: open map}, the set \[ E:= \{ (w_1e^{-t_1z}, \ldots, w_me^{-t_mz}): z\in \C,~ \wbar \in \C^m \text{with } p(\wbar) = 0 \}\] is open in $\C^m$ and by the above claim it intersects $\s_1^m$. Therefore, there is $l\in \Z$ such that $( e^{-t_1(\hat{z}+2\pi i l)}, \ldots, e^{-t_m(\hat{z}+2\pi i l)})\in E$.

Thus, for some $\tilde{\wbar}\in \C^m$ with $p(\wbar)=0$ and for some $\tilde{z}\in \C$ we have
\[ e^{-t_k(\hat{z}+2\pi i l)} = \tilde{w}_ke^{-t_k\tilde{z}} \text{ for all } k. \]
This implies $\tilde{w}_k = e^{t_k(\tilde{z}-\hat{z}-2\pi i l)}$ which means that $\tilde{z}-\hat{z}-2\pi i l$ is a solution to \eqref{eqn: p(e^t_1z,...) = 0}.

\section{The Exponential Closedness conjecture}

Theorem \ref{thm: exp sums one var} is actually an instance of an important open problem, the \textit{Exponential Closedness} conjecture. The idea behind this conjecture is that systems of equations involving polynomials and exponentials should always have a solution in the complex numbers, unless there is a \textit{good reason} for having no solutions. We will now introduce the ingredients in the statement of this conjecture, highlighting why these should represent ``good reasons'' for a system of equations not to be solvable in $\mathbb{C}$. We will then survey some known partial results.

\subsection{Intepreting systems by varieties}

We will deal with systems of exponential polynomial equations by looking at ``exponential points'' on complex algebraic varieties.

Exponential polynomials are expressions obtained by iterating the field operations and exponentiation, e.g. $e^{z_2e^{z_1^2}}+5e^{z_3^6+\pi z_2}-3z_1z_3^6$. However, when considering equations, we may simplify a system of such equations to another system (possibly in more variables) where exponentiation only occurs in the form $e^{z_k}$ where $z_k$ is a variable. For instance, the equation $e^{z_2e^{z_1^2}}+5e^{z_3^6+\pi z_2}-3z_1z_3^6=0$ is equivalent to the following system
\[
\begin{cases}
    z_4 = z_1^2\\
    z_5 = z_2e^{z_4}\\
    z_6 = z_3^6+\pi z_2\\
    e^{z_5}+5e^{z_6}-3z_1z_3^6=0.
    \end{cases}
\]

In general, given an equation  \[ f\left(z_1,\dots,z_n,e^{z_1},\dots,e^{z_n},e^{e^{z_1}}, \dots, e^{e^{z_n}}, \dots, e^{e^{\iddots^{e^{z_1}}}}, \dots, e^{e^{\iddots^{e^{z_n}}}}\right)=0 \] where each exponential is iterated at most $k-1$ times and $f$ is a polynomial in $nk$ variables, we can get rid of iterated exponentials by adding variables: for example, we may replace the previous equation by the following system in $2nk$ variables:
    \[\begin{cases}
    f\left( z_1,\dots,z_{nk} \right)=0 \\
    z_{n+j}=z_{nk+j} & \textnormal{ for }j=1,\dots,n(k-1)\\
     z_{nk+i}=e^{z_i} & \textnormal{ for }i=1,\dots,nk .
\end{cases}\]

We can similarly add new variables for polynomials appearing inside the exponential function, as in the example mentioned earlier. Hence, we may reduce every system of equations in polynomials and exponentials as one given by some polynomial equations in $2n$ variables $z_1,\dots,z_{2n}$, together with the equations $z_{n+i}=e^{z_i}$. The first set of equations then defines an algebraic subvariety $V$ of $\mathbb{C}^{2n}$, and the second set of equations determines a point of the form $(z_1,\dots,z_n,e^{z_1},\dots,e^{z_n})$ on $V$; this is what we will call an \textit{exponential point} on $V$.
 
To make better geometric sense of the statement of the conjecture, it can be helpful to see $V$ as an algebraic subvariety of $\mathbb{C}^n \times (\Cx)^n$, rather than one of $\mathbb{C}^{2n}$, since the exponential map is a homorphism from the additive group of complex numbers onto the multiplicative group of \textit{non-zero} complex numbers. We will take this approach. 

\subsection{Freeness and rotundity}

If $Q \leq \mathbb{C}^n$ is a linear subspace defined over $\mathbb{Q}$, then $\exp(Q)$ is an algebraic subgroup of $(\Cx)^n$. For example, if $Q=\{(z_1,z_2) \in \mathbb{C}^2 \mid z_1-2z_2=0\}$, then $\exp(Q)=\{(z_1,z_2) \in (\Cx)^2 \mid z_1z_2^{-2}=1\}$. Hence, the quotient $\pi_Q:\mathbb{C}^n \times (\Cx)^n \twoheadrightarrow \mathbb{C}^n /Q \times (\Cx)^n/\exp(Q)$ is an algebraic map.

\begin{definition}
     Let $V \subseteq \mathbb{C}^n  \times (\Cx)^n$ be an algebraic subvariety.

    We say that $V$ is \textit{vertical} (resp. \textit{horizontal}) if it has the form $\{\zbar\} \times W$ for some $\zbar \in \mathbb{C}^n$ and some subvariety $W \subseteq (\Cx)^n$ (resp. if it has the form $X \times \{\wbar\}$ for some $\wbar\in (\Cx)^n$ and some subvariety $X \subseteq \mathbb{C}^n$). 
\end{definition}

It should be clear why we cannot in general expect exponential points on horizontal or vertical varieties: if $V=\{\zbar\} \times W$ is vertical, for example, then as soon as $\exp(\zbar) \notin W$ we have no hope of finding an exponential point in $V$.

If a variety is not horizontal or vertical, but there is a linear subspace $Q \leq \mathbb{C}^n$ defined over $\mathbb{Q}$ such that $\pi_Q(V)$ is horizontal or vertical, then we risk running into the same issue. Hence we give the following definition, which rules out this type of behaviour.

\begin{definition}
    Let $V \subseteq \mathbb{C}^n \times (\Cx)^n$ be an algebraic subvariety.

    We say $V$ has \textit{no vertical projections} (resp. \textit{no horizontal projections}) if for every linear subspace $Q \leq \mathbb{C}^n$ defined over $\mathbb{Q}$, $\pi_Q(V)$ is not vertical (resp. horizontal).

    If $V$ has no vertical and no horizontal projections, then we say $V$ is \textit{free}.
\end{definition}

However, this does not rule out all bad behaviours: the dimension of $V$ should also play a role. Indeed, if $V$ is a product of curves $C_1 \subseteq \mathbb{C}^3$ and $C_2 \subseteq (\Cx)^3$, even if it is free we cannot really expect an exponential point to exist: each of $C_1$ and $C_2$ is defined by two equations, and we need three conditions to define an exponential point, so we end up with a system of seven equations in six variables.

Intuitively, we expect that $V$ should have dimension at least $n$, to avoid having an overdetermined system. Furthermore, just as above, we need to eliminate this bad behaviour in the quotients.

\begin{definition}
    Let $V \subseteq \mathbb{C}^n \times (\Cx)^n$ be an algebraic subvariety.

    We say $V$ is \textit{rotund} if for every linear subspace $Q \leq \mathbb{C}^n$ defined over $\mathbb{Q}$, $\dim \pi_Q(V) \geq n- \dim Q$.
\end{definition}

We can now state the conjecture.

\begin{conjecture}[Exponential Closedness]\label{conjecture:eac}
    Let $V \subseteq \mathbb{C}^n \times (\Cx)^n$ be an algebraic subvariety. If $V$ is free and rotund, then $V$ contains an exponential point.
\end{conjecture}

\begin{remark}
    Usually in applications one needs the existence of a Zariski dense set of exponential points on $V$, rather than a single exponential point. In other words, a seemingly stronger form of the conjecture states: if $V \subseteq \mathbb{C}^n \times (\Cx)^n$ is free and rotund and $V' \subsetneq V$ is a proper algebraic subvariety then there is an exponential point in $V \setminus V'$.

    Conjecture \ref{conjecture:eac} as stated actually implies this statement. Indeed, if $V' \subseteq V$ is a proper algebraic subvariety then there is a non-zero algebraic function $F:V \rightarrow \mathbb{C}$ such that $F(\vbar)=0$ for each $\vbar \in V'$. We can then consider the algebraic subvariety $V_1$ of $\mathbb{C}^{n+1} \times (\Cx)^{n+1}$ defined by $$ (z_1,\dots,z_n,z_{n+2},\dots,z_{2n+1}) \in V \textnormal{ and } z_{2n+2}=F(z_1,\dots,z_n,z_{n+2},\dots,z_{2n+1}).$$

    It is not hard to show that if $V$ is free and rotund, then $V_1$ is as well. If Conjecture \ref{conjecture:eac} is true, then $V_1$ contains an exponential point, that is, a point $(z_1,\dots,z_{n+1}, e^{z},\dots, e^{z_{n+1}})$ such that $\vbar:=(z_1,\dots,z_n,e^{z_1},\dots, e^{z_n})$ is an exponential point on $V$ and that $e^{z_{n+1}}=F(\vbar)$. In particular, $F(\vbar) \neq 0$, so that $\vbar \notin V'$.
\end{remark}

\subsection{Raising to powers}

Theorem \ref{thm: exp sums one var} is a special case of a more general result, which is in turn a special case of Conjecture \ref{conjecture:eac}.

\begin{theorem}{\cite[Theorem 8.8]{Gallinaro-exp-sums-trop}\label{thm: raising to powers}}
    Let $L \leq \mathbb{C}^n$ and $W \subseteq (\Cx)^n$ be algebraic subvarieties. If $L \times W$ is free and rotund, then $L \times W$ contains an exponential point.
\end{theorem}

Indeed, if $L:=\left\{(r_1z,\dots,r_nz) \in \mathbb{C}^{n} \mid z \in \mathbb{C} \right\}$ for some real $r_n > \dots > r_1$ and $W:=\left\{ \left(w_1,\dots,w_n \right) \in (\Cx)^n \mid \sum_{k=1}^n c_kw_k=0 \right\}$ with $c_n \neq 0$, then the existence of an exponential point on $L \times W$ corresponds to the existence of a solution to $\sum_{k=1}^n c_ke^{r_k z}$, as given by Theorem \ref{thm: exp sums one var}. In this case, freeness and rotundity are simpler to define, and to verify: we invite the reader to verify the following statements. Proposition \ref{proposition: ax-schanuel type statement} will be very helpful.

\begin{itemize}
    \item[(1)] $L \times W$ has no vertical projections if and only if the real numbers $r_1,\dots,r_n$ are $\mathbb{Q}$-linearly independent, if and only if $L$ is not contained in any linear subspace of $\mathbb{C}^{n}$ defined over $\mathbb{Q}$.
    \item[(2)] $L \times W$ has no horizontal projections if and only if the polynomial $p$ does not define a coset of an algebraic subgroup of $(\Cx)^{n}$. If $p$ is irreducible, this holds if and only if $p$ has more than two monomials.
    \item[(3)] $L \times W$ is rotund if and only if, for every $a \in \Cx$, the function $\sum_{k=1}^n c_k a e^{r_kz}$ is not identically zero. 
\end{itemize}

The proof of Theorem \ref{thm: raising to powers} is divided into two cases, according to whether or not $L$ is defined over $\mathbb{R}$.

If $L$ is defined over $\mathbb{R}$ then freeness of $L \times W$ (or more precisely, the fact that $L \times W$ has no vertical projections) implies that $\exp(L)$ is dense in $\exp(L) \cdot \mathbb{S}_1^n$: this is the multi-dimensional analogue of Proposition \ref{prop: density}. This, together with a multi-dimensional analogue of Proposition \ref{prop: open map}, allows us to reduce the problem to finding a point in $(\exp(L) \cdot \mathbb{S}_1^n) \cap W$: in other words, a point $(w_1,\dots,w_n) \in W$ such that $(\log|w_1|,\dots,\log|w_n|)$ is the real part of a point in $L$. To do this, one needs a result of Khovanskii on the existence of solutions to exponential sums equations which satisfy a condition known as ``non-degeneracy at infinity''. We do not state this condition here; let us just say, informally, that it means that certain systems of equations which approximate the given one at infinity have no solutions; and that if $L \times W$ is rotund, then there is $\sbar \in \mathbb{S}_1^n$ such that the system associated to $L \times (\sbar \cdot W)$ is non-degenerate at infinity. The tools to verify that we can apply Khovanskii's result come from an area of mathematics known as \textit{tropical geometry}, which studies among other things the behaviour of algebraic subvarieties of $(\Cx)^n$ as their points approach 0 or $\infty$.

The case in which $L$ is not defined over $\mathbb{R}$ is more complicated, as in that case it is no longer true that $\exp(L)$ is dense in $\exp(L) \cdot \mathbb{S}_1^n$ (it may even be a closed analytic subgroup of $(\Cx)^n$: for example, this is the case for $L:=\{(z,iz) \mid z \in \mathbb{C} \}$). We need a previous result on exponential sums, due to Kazarnovskii, to find a sequence in $W$ of points which are not getting too far from $\exp(L)$. Then we use more tools from tropical geometry, which tell us that as we follow this sequence on $W$, it is possible to approximate $W$ locally better and better by an algebraic subvariety $W'$  of $(\Cx)^n$ which has non-trivial multiplicative stabilizer and such that $L \times W'$ is still rotund: taking the quotient by this stabilizer allows us to perform an induction and show that $\exp(L) \cap W' \neq \varnothing$, and that these approximate solutions can be ``lifted'' to points in $\exp(L) \cap W$.  

\subsection{Other known results} Using standard methods from complex analysis, it is possible to prove Conjecture \ref{conjecture:eac} in the case of $n=1$.

\begin{theorem}\label{thm: ec n=1}
    Let $V \subseteq \mathbb{C} \times \mathbb{C}^\times$ be an algebraic subvariety. If $V$ is free and rotund, then $V$ contains an exponential point.
\end{theorem}

\begin{proof}
    We assume $V$ is irreducible; then there is an irreducible polynomial $f \in \mathbb{C}[X,Y]$ such that $V=\{(z,w) \in \mathbb{C} \times \mathbb{C}^\times \mid f(z,w)=0 \}.$ An exponential point in $V$ is the a zero of the entire function $f(z,e^z)$. By the Hadamard factorisation theorem, if $f(z,e^z)$ has no zeros then there is $q \in \mathbb{C}[z]$ such that $f(z,e^z)=e^{q(z)}$. Then $f(X,Y)=Y^{q(X)}$; clearly this is only possible if $q(X)$ is constant, so $f(X,Y)$ does not depend on $X$. This violates freeness of $V$.
\end{proof}

This is a well-known result and can be found, for example, in \cite{marker-remark-pseudoexp} where Marker proves a stronger result, again in the case $n=1$.

Theorem~\ref{thm: ec n=1} can be generalised to any exponential polynomial equation in a single variable using a theorem of Henson and Rubel \cite{henson-rubel} stating that if an exponential polynomial has no zeroes then it is equal to the exponential of another exponential polynomial.

Theorem~\ref{thm: ec n=1} has also been generalised to the following cases of Conjecture \ref{conjecture:eac}. Let $\pi_1:\mathbb{C}^n \times (\Cx)^n \twoheadrightarrow \mathbb{C}^n$ denote the projection to the first $n$ coordinates.

\begin{theorem}{\cite[Proposition 2]{brown-masser}}
     Let $V \subseteq \mathbb{C}^n \times (\Cx)^n$ be an algebraic subvariety. If $\dim \pi_1(V)=n$, then $V$ has an exponential point.
\end{theorem}

\begin{theorem}{\cite[Theorem 1.1]{mantova-masser}}
     Let $V \subseteq \mathbb{C}^n \times (\Cx)^n$ be an algebraic subvariety. If $V$ has no vertical projections and $\dim \pi_1(V)=1$, then $V$ has an exponential point.
\end{theorem}

Together, these results imply Conjecture \ref{conjecture:eac} for all subvarieties of $\mathbb{C}^2 \times (\Cx)^2$.

\subsection{Relation to Schanuel's conjecture}

We conclude the paper with a brief discussion about the relation between Conjecture \ref{conjecture:eac} and Schanuel's conjecture, a famous open problem in transcendental number theory.

\begin{conjecture}
    Let $z_1,\dots,z_n$ be $\mathbb{Q}$-linearly independent complex numbers.

    Then $\textnormal{trdeg}_{\mathbb{Q}}(z_1,\dots,z_n,e^{z_1},\dots,e^{z_n}) \geq n$.
\end{conjecture}

There are no known implications between Schanuel's conjecture and Exponential Closedness (and it is not expected that there are any). However, these conjectures are often discussed together because of a certain ``duality'' between them: Schanuel's conjecture can be seen as saying that if $V \subseteq \mathbb{C}^n \times (\Cx)^n$ is an algebraic subvariety with $\dim V <n$, defined over $\mathbb{Q}^{\alg}$, and $\vbar$ is an exponential point in $V$, then $\vbar$ is not generic in $V$ over $\mathbb{Q}^{\alg}$, unless $V$ itself is contained in a coset of a proper algebraic subgroup of $\mathbb{C}^n \times (\Cx)^n$.

Hence, while Exponential Closedness gives sufficient conditions for the existence of exponential points, Schanuel's conjecture gives a necessary condition for an exponential point in a variety defined over $\mathbb{Q}^{\alg}$ to be generic over $\mathbb{Q}^{\alg}$. In other words, Exponential Closendess somehow gives an ``upper bound'' on the transcendence of the exponential function, predicting that given algebraic relations of a certain form there need to be exponential points that satisfy them; while Schanuel's conjecture gives a ``lower bound'', predicting that exponential points are never going to satisfy too many algebraic relations.

We have already mentioned a form of Schanuel's conjecture for function fields that was established by Ax, see Theorem \ref{Axsch}. A complex version of Ax's theorem states that if $V \subseteq \mathbb{C}^n \times (\Cx)^n$ is an algebraic subvariety, $\Gamma_{\exp}$ is the graph of the exponential function $\exp:\mathbb{C}^n \rightarrow (\Cx)^n$, and $V \cap \Gamma_{\exp} \neq \varnothing$, then the dimension of the complex analytic set $V \cap \Gamma_{\exp}$ is precisely $\dim V - n$, unless some irreducible component of it is contained in a coset of an algebraic subgroup of $\mathbb{C}^n \times (\Cx)^n$. This also suggests some sort of duality with Exponential Closedness: Ax's Theorem says that if an intersection exists, then it should have the expected dimension, modulo the behaviour of certain algebraic subgroups; Exponential Closedness predicts that if the dimension of $V$ is big enough for us to expect that $V$ intersects $\Gamma_{\exp}$, then modulo the behaviour of certain algebraic subgroups the intersection is indeed non-empty.

\bibliographystyle{alpha}
\bibliography{ref}

\begin{thebibliography}{MM24}

\bibitem[Ax71]{Ax}
James Ax.
\newblock On {S}chanuel's conjectures.
\newblock {\em Annals of Mathematics}, 93:252--268, 1971.

\bibitem[BM17]{brown-masser}
Dale Brownawell and David Masser.
\newblock Zero estimates with moving targets.
\newblock {\em J. Lond. Math. Soc.}, 95(2):441--454, 2017.

\bibitem[Gal23]{Gallinaro-exp-sums-trop}
Francesco Gallinaro.
\newblock Exponential sums equations and tropical geometry.
\newblock {\em Sel. Math. New Ser.}, 29(49), 2023.

\bibitem[HR84]{henson-rubel}
C.~Ward Henson and Lee~A. Rubel.
\newblock Some applications of {N}evanlinna theory to mathematical logic: identities of exponential functions.
\newblock {\em Trans. Amer. Math. Soc.}, 282(1):1--32, 1984.

\bibitem[HW80]{hardy-wright}
Godfrey Hardy and Edward Wright.
\newblock {\em An Introduction to the Theory of Numbers}.
\newblock Oxford University Press, 1980.

\bibitem[Mar06]{marker-remark-pseudoexp}
David Marker.
\newblock A remark on zilber's pseudoexponentiation.
\newblock {\em The Journal of Symbolic Logic}, 71(3):791--798, 2006.

\bibitem[MM24]{mantova-masser}
Vincenzo Mantova and David Masser.
\newblock Polynomial-exponential equations -- some new cases of solvability, 2024.
\newblock To appear in \textit{Proc. Lond. Math. Soc}.

\bibitem[Zil02]{Zilb-exp-sum-published}
Boris Zilber.
\newblock Exponential sums equations and the {S}chanuel conjecture.
\newblock {\em J.L.M.S.}, 65(2):27--44, 2002.

\bibitem[Zil05]{Zilb-pseudoexp}
Boris Zilber.
\newblock Pseudo-exponentiation on algebraically closed fields of characteristic zero.
\newblock {\em Annals of Pure and Applied Logic}, 132(1):67--95, 2005.

\bibitem[Zil15]{ZilbExpSum}
Boris Zilber.
\newblock The theory of exponential sums.
\newblock arXiv:1501.03297, 2015.

\end{thebibliography}

\end{document}